\documentclass[graybox]{svmult}

\usepackage{mathptmx}       
\usepackage{helvet}         
\usepackage{courier}        
\usepackage{type1cm}        
\usepackage{makeidx}         
\usepackage{graphicx}        
\usepackage{multicol}        
\usepackage[bottom]{footmisc}

\def\A{{\mathbf A}}

\def\B{{\mathbf B}}

\def\N{{\mathbf N}}

\def\Q{{\mathbf Q}}

\def\R{\mathbf R}

\def\Z{{\mathbf Z}}
\def\Z{\mathbf Z}

\def\limproj{\mathop{\oalign{lim\cr\hidewidth$\longleftarrow$\hidewidth\cr}}}

\makeindex             

\begin{document}

\title*{Two points of the boundary of toric geometry}
\author{Bernard Teissier}

\institute{ \at Institut math\'ematique de Jussieu-Paris Rive Gauche, \email{bernard.teissier@imj-prg.fr}
}
%
%
\maketitle
\hfill\textit{To Antonio Campillo, on the occasion of his 65th Birthday}\par\medskip
\abstract{ This note presents two observations which have in common that they lie at the boundary of toric geometry. The first one because it concerns the deformation of affine toric varieties into non toric germs in order to understand how to avoid some ramification problems arising in the study of local uniformization in positive characteristic, and the second one because it uses limits of projective systems of equivariant birational maps of toric varieties to study the space of additive preorders on $\Z^r$ for $r\geq 2$.}

\section{Using toric degeneration to avoid wild ramification}
\label{sec:1}
In his book \cite{C}, Campillo introduces and studies a notion of characteristic exponents for plane branches over an algebraically closed field of positive characteristic. One of the definitions he gives is that the characteristic exponents are those of a plane branch in characteristic zero having the same process of embedded resolution of singularities. His basic definition is given in terms of Hamburger-Noether expansion\footnote{The Hamburger-Noether expansion is an algorithm extracting in any characteristic a description of the resolution process by point blowing-ups from a parametric representation $x(t),y(t)$.} but we shall not go into this here. He then gives an example to show that even if you do have a Puiseux-type parametrization for a branch in positive characteristic, which is not always the case, the exponents you see in the parametrization are \textit{not} in general the characteristic exponents. His example (see \cite[Chap. III, \S5, Example 3.5.4]{C}) is this:\par\noindent Let $p$ be a prime number. Let us choose a field $k$ of characteristic $p$ and consider the plane branch defined parametrically by $x=t^{p^3},\ y=t^{p^3+p^2}+t^{p^3+p^2+p+1}$, and implicitly by a unitary polynomial of degree $p^3$ in $y$ with coefficients in $k[[x]]$. Campillo computes by a Hamburger-Noether expansion the characteristic exponents and finds $\beta_0=p^3,\beta_1=p^3+p^2,\beta_2=p^3+2p^2+p,\beta_3=p^3+2p^2+2p+1$. Note that there are four characteristic exponents while the parametrization exhibits three exponents in all. In \cite[Remark 7.19]{Te1}, the author had computed directly from the parametrization the generators of the semigroup of values of the $t$-adic valuation of $k[[t]]$ on the subring $k[[t^{p^3},t^{p^3+p^2}+t^{p^3+p^2+p+1}]]$ and found the numerical semigroup with minimal system of generators:$$\Gamma= \langle p^3,\  p^3+p^2,\ p^4+p^3+p^2+p, \ p^5+p^4+p^3+p^2+p+1\rangle.$$ We can verify that Campillo's characteristic exponents and the generators of the semigroup satisfy the classical relations of Zariski, in accordance with \cite[Proposition 4.3.5]{C}.\par
This example is very interesting because it shows that in positive characteristic, even if a Puiseux-type expansion exists, its exponents do not determine the resolution process. In thinking about resolution in positive characteristic, one should keep away from ideas inspired by Puiseux exponents. The semigroup, however, does determine the resolution process in all characteristics. It is shown in \cite{Te1} and \cite{Te3} that in the case of analytically irreducible curves one can obtain embedded resolution by studying the embedded resolution of the monomial curve corresponding to the minimal system of generators of the numerical semigroup $\Gamma$ of the values taken on the algebra of the curve by its unique valuation. Another, more classical, reason is that the semigroup determines the Puiseux exponents of the curve in characteristic zero having the same resolution process, and therefore this resolution process (see \cite[Chap. IV, \S 3]{C}).\par 
The polynomial $f(x,y)\in k[[x]][y]$ defining our plane branch can be obtained by eliminating $u_2,u_3$ between three equations which are: $$y^p-x^{p+1}-u_2=0,\  u_2^p-x^{p(p+1)}y-u_3=0,\  u_3^p-x^{p^2(p+1)}u_2=0.$$This makes it apparent that our plane branch is a flat deformation of the curve defined by $y^p-x^{p+1},\  u_2^p-x^{p(p+1)}y=0,\  u_3^p-x^{p^2(p+1)}u_2=0$, which is the monomial curve $C^\Gamma$ in $\A^4(k)$ given parametrically by $x=t^{p^3},y=t^{p^3+p^2},u_2=t^{p^4+p^3+p^2+p},u_3=t^{p^5+p^4+p^3+p^2+p+1}$. That is, the monomial curve whose affine algebra is the semigroup algebra $k[t^\Gamma]$. The binomial equations correspond to a system of generators of the arithmetical relations between the generators of the semigroup. Compare with \cite[Remark 7.19]{Te3}; we have chosen here the canonical system of relations between the generators of the semigroup, where all the exponents in the second monomial of the equations except the exponent of $x$ are $<p$.\par\noindent 
Moreover, if we give to each of the variables $x,y,u_2,u_3$ a weight equal to the exponent of $t$ for this variable in the parametrization of $C^\Gamma$, we see that to each binomial is added a term of higher weight. This is an \textit{overweight deformation} of a prime binomial ideal in the sense of \cite[\S 3]{Te3}. \par
Eliminating $u_2$ and $u_3$ between the three equations gives the equation of our plane curve with semigroup $\Gamma$:
$$(y^p-x^{p+1})^{p^2}-2x^{p^2(p+1)}y^p+x^{(p^2+1)(p+1)}=0.$$
If $p=2$ this reduces to $(y^2-x^3)^4-x^{15}=0$, which looks like it should be the overweight deformation $y^2-x^3-u_2=0, u_2^4-x^{15}=0$ of the monomial curve with equations $y^2-x^3=0, u_2^4-x^{15}=0$. However, such is not the case because the two binomials $y^2-x^3,u_2^4-x^{15}$ do not generate a prime ideal in $k[[x,y,u_2]]$. One verifies that $u_2^2-x^6y$ is not in the ideal but its square is. The binomials do not define an integral monomial curve, but in fact the non reduced curve given parametrically by $x=t^8,y=t^{12},u_2=t^{30}$ with non coprime exponents, while we know that the semigroup of our curve is $\Gamma=\langle 8,12,30,63\rangle$. The equation of our plane curve is indeed irreducible, but it is \textit{not} an overweight deformation of a integral monomial curve in $\A^3(k)$. One has to embed our plane curve in $\A^4(k)$ to view it as an overweight deformation of a monomial curve.\par\noindent The reader can verify that the same phenomenon occurs in any positive characteristic. For $p\neq 2$ our curve looks like an overweight deformation of the curve in $\A^3(k)$ defined by the binomials $y^p-x^{p+1}, u_2^{p^2}-2x^{p^2(p+1)}y^p$ but these binomials do not generate a prime ideal: the binomial $u_2^p-2x^{p(p+1)}y$ is not in the ideal, but its $p$-th power is by Fermat's little theorem. So appearances can be deceptive also from the equational viewpoint.\par
We expect the same result in characteristic zero, and this is true as soon as the field $k$ contains the $p$-th roots of 2. In this case, by (the proof of) \cite[Theorem 2.1]{E-S} the ideal $I$ generated by $y^p-x^{p+1}, u_2^{p^2}-2x^{p^2(p+1)}y^p$ is not prime because the lattice ${\mathcal L}$ in $\Z^3$ generated by $(-(p+1),p,0)$ and $(-p^2(p+1),-p,p^2)$ is not saturated; the vector $w=(-p(p+1),-1,p)$ is not in ${\mathcal L}$, but $pw$ is. Here the argument is that if $I$ was a prime ideal, at least one of the factors of the product $\prod_{\zeta^p=2}(u_2^p-\zeta x^{p(p+1)}y)=u_2^{p^2}-2x^{p^2(p+1)}y^p$ should be in $I$, which is clearly impossible.\par
This raises the following question: Given an algebraically closed field $k$, assuming that we have a Puiseux expansion $x=t^n,y=\sum_ja_jt^j,\ a_j\in k$, can one predict, from the set $\{j/a_j\neq 0\}$, useful information about the semigroup of the corresponding plane curve, even only a bound on the number of generators? More generally, denoting by $k[[t^{\Q_{\geq 0}}]]$ the Hahn ring of series whose exponents are non negative rational numbers forming a well-ordered set, and given a series $y(t)\in k[[t^{\Q_{\geq 0}}]]$ which is integral over $k[[t]]$, can one deduce from the knowledge about the exponents and coefficients of $y(t)$ provided by the work of Kedlaya in \cite{Ke} useful information about the semigroup of the plane branch whose equation is the integral dependence relation?
\par\noindent
This example has another interesting feature. We note that no linear projection of our plane curve to a line in the plane can be tamely ramified. However, our curve is a deformation of the monomial curve, for which the projection to the $u_3$-axis is tamely ramified. More precisely, the inclusion of $k$-algebras $k[u_3]\subset k[t^\Gamma]$ corresponding to $u_3\mapsto t^{p^5+p^4+p^3+p^2+p+1}$ gives rise to an extension $\Z\subset \Z$ of value groups, for the $u_3$-adic and $t$-adic valuations respectively, whose index is $p^5+p^4+p^3+p^2+p+1$ and hence prime to $p$, while there is no residue field extension. This is related to what we saw in the first part of this section; it is precisely the coordinate $u_3$ missing in $\A^3(k)$ which provides the tame projection.\par\noindent
It is a general fact that if $\Gamma\subset \Z^r$ is a finitely generated semigroup generating $\Z^r$ as a group, given a system of generators $\Gamma=\langle \gamma_1,\ldots ,\gamma_N\rangle$ and an algebraically closed field $k$, there always exist $r$ of the generators, say $\gamma_{i_1},\ldots ,\gamma_{i_r}$ such that the inclusion of $k$-algebras $k[u_{i_1},\ldots ,u_{i_r}]\subset k[t^\Gamma]$ defined by $u_{i_\ell}\mapsto t^{\gamma_{i_\ell}}$ defines a tame extension of the fraction fields. In fact, the corresponding map $\pi\colon{\rm Spec}k[t^\Gamma]\rightarrow\A^r(k)$ is \'etale on the torus of ${\rm Spec}k[t^\Gamma]$. Note that the subset $\{i_1,\ldots ,i_r\}\subset \{1,\ldots ,N\}$ depends in general on the characteristic of the field $k$. This is immediately visible in the case of a monomial curve, where the result follows directly from the fact that the generators of the semigroup are coprime since the semigroup generates $\Z$ as a group.\par\noindent In the general case, as explained in \cite[Proof of Proposition 3.20 and Proof of Proposition 7.4]{Te3}, modulo a Gale-type duality this fact is an avatar of the fact that the \textit{relative torus} ${\rm Spec}\Z[t^{\Z^r}]$ of 
${\rm Spec}\Z[t^\Gamma]$ is smooth over ${\rm Spec}\Z$ while of course the whole toric variety is not. This implies, for each prime number $p$, the non-vanishing on the torus of certain jacobian determinants which are exactly those whose non vanishing is needed to ensure the \'etaleness of the map $\pi$. We note that the map $\pi$ is not finite in general; it is finite if and only if the vectors $\gamma_{i_1},\ldots ,\gamma_{i_r}$ generate the cone of $\R^r$ generated by $\Gamma$.\par\noindent More precisely, recall the description found in \cite{Te1}, before Prop. 6.2, of the jacobian ideal of an $r$-dimensional affine toric variety defined by a prime binomial ideal $ P\subset k[U_1,\ldots
,U_N]$. The jacobian determinant
$J_{G,\mathbf L'}$ of rank
$c=N-r$ of the generators $(U^{m^\ell}-\lambda_\ell U^{n^\ell})_{\ell\in\{1,\ldots , L\}}$ of $P$, associated to a sequence $G=(k_1,\ldots , k_c)$ of distinct elements of $\{1,\ldots, N\}$ and a subset $\mathbf L'\subseteq \{1,\ldots , L\}$ of
cardinality $c$, satisfies the congruence
$$U_{k_1}\ldots U_{k_c}.J_{G,\mathbf L'}\equiv
\big(\prod_{\ell\in \mathbf L'} U^{m^\ell}\big)\hbox{\rm Det}_{G,\mathbf L'}(\langle m-n\rangle ) \ \
\hbox{\rm mod.} P,$$ where $\big(\langle m-n\rangle\big)$ is the matrix of the
vectors $(m^\ell-n^\ell)_{\ell\in \{1,\ldots , L\}}$,  and $\hbox{\rm Det}_{G,\mathbf L'}$ indicates the minor
in question. If the field $k$ is of characteristic $p$, choosing a minor which is not divisible by $p$ amounts to choosing $r$ of the coordinates such that the corresponding projection to $\A^r(k)$ of a certain binomial variety containing the toric variety as one of its irreducible components (see \cite[Corollary 2.3]{E-S}) is \'etale outside of the coordinate hyperplanes.\par
It is shown in \cite[Proof of Proposition 3.20]{Te3} that for any prime $p$ there exist minors ${\rm Det}_{G,\mathbf L'}(\langle m-n\rangle )$ which are not divisible by $p$. As mentioned above, this is the equational aspect of the smoothness over ${\rm Spec}\Z$ of the torus ${\rm Spec}\Z[t^{\Z^r}]$ of the affine toric variety over $\Z$ corresponding to the subsemigroup $\Gamma$ of $\Z^r$.\par
The next step is to realize that an overweight deformation preserves the non vanishing. Let us illustrate this on our example:\par\noindent
The jacobian matrix of our three binomial equations over a field of characteristic $p$ is $$\pmatrix{-x^p &0&0&0\cr
0&-x^{p(p+1)}&0&0\cr 0&0&-x^{p^2(p+1)}&0\cr}$$
We see that there is only one minor which is non zero on the torus, where $x\neq 0$, corresponding to the inclusion $k[u_3]\subset k[t^\Gamma]$. After the overweight deformation which produces our plane branch, the jacobian matrix becomes
$$\pmatrix{-x^p &0&-1&0\cr
0&-x^{p(p+1)}&0&-1\cr 0&0&-x^{p^2(p+1)}&0\cr}$$
and we see that the same minor is $\neq 0$. The facts we have just mentioned imply that for each characteristic $p$, some of these minors are non zero modulo $p$. But the duality implies that if $u_{i_1},\ldots ,u_{i_r}$ are the variables not involved in the derivations producing one such minor, then the absolute value of this minor, which is not divisible by $p$, is the index of the extension of groups $\Z\gamma_{i_1}\oplus\cdots \oplus\Z\gamma_{i_r}\subset \Z^r$. This explains the tameness result we have just seen, and also why it is preserved by overweight deformation. In our example the matrix $\big(\langle m-n\rangle\big)$ after reduction modulo $p$ is
$$\pmatrix{-1 &0&0&0\cr
0&-1&0&0\cr 0&0&-1&0\cr}$$
In conclusion the method which we can apply to our curve, which is that if we embed it in the affine space spanned by the monomial curve associated to its semigroup, in any characteristic we are certain to find a tame projection to a coordinate axis, will work for any finitely generated semigroup $\Gamma$ in $\Z^r$ generating $\Z^r$ as a group and provide tame projections to $\A^r(k)$ of the affine toric variety ${\rm Spec}k[t^\Gamma]$, which will remain tame after an overweight deformation. In order to obtain tame projections for the space whose valuation we want to uniformize, we have to first re-embed it in the space where the associated toric variety lives so that it can appear as an overweight deformation. This is an important element in the proof of local uniformization for rational Abhyankar valuations (those with value group $\Z^{{\rm dim}R}$) on equicharacteristic excellent local domains $R$ with an algebraically closed residue field given in \cite{Te3}.\par\noindent The reason is that the local uniformization of Abhyankar valuations on an equicharacteristic  excellent local domain with algebraically closed residue field can be reduced to that of rational valuations of complete local domains whose semigroup is a finitely generated subsemigroup of $\Z^r$ and then the complete local domain is an overweight deformation of an affine toric variety by \cite[Proposition 5.1]{Te3}.

\section{Additive preorders and orders on $\Z^r$ and projective limits of toric varieties}
\label{sec:2}
In this section the space of additive preorders on $\Z^r$ with a topology defined by Kuroda and Sikora (see \cite{Ku}, \cite{S}) is presented as homeomorphic to the projective limit of finite topological spaces which are spaces of orbits on toric varieties with the topology induced by the Zariski topology, and the space of additive orders as the closed subspace corresponding to the projective limit of the subsets of closed points of the preceding finite topological spaces. We use this to give a toric proof of a theorem of Sikora in \cite{S} showing that the space of additive orders on $\Z^r$ with $r\geq 2$ is homeomorphic to the Cantor set.  The relation between toric geometry and preorders on $\Z^r$ is due to Ewald-Ishida in \cite{E-I} where they build the analogue in toric geometry of the Zariski-Riemann space of an algebraic variety. This relation was later developed by Pedro Gonz\'alez P\'erez and the author in \cite{GP-T}, which contains in particular the results used here. In that text we quoted Sikora's result without realizing that we could give a direct proof in our framework. We begin with the:
\begin{proposition}\label{or} Let $I$ be a directed partially ordered set and $(b_{\iota',\iota}\colon X_{\iota'}\to X_\iota$ for $\iota'>\iota)$ a projective system indexed by $I$ of surjective continuous maps between finite topological spaces. Then:\par\noindent
a) The projective limit $\mathcal X$ of the projective system, endowed with the projective limit topology, is a quasi-compact space.\par\noindent
b) In each $X_\iota$ consider the subset $D_\iota$ of closed points, on which the induced topology is discrete. If we assume that:\begin{enumerate}
 \item The maps $b_{\iota',\iota}$ map each $D_{\iota'}$ onto $D_\iota$ in such a way that the inverse image in the projective limit $\mathcal D$ of the $D_\iota$ by the canonical map  $b_{\infty,\iota}\colon{\mathcal D}\to D_\iota$ of any element of a $D_\iota$ is infinite.
 \item There exists a map $h\colon I\to\N\setminus\{0\}$ such that $h(\iota')\geq h(\iota)$ if $\iota'>\iota$ and $h^{-1}([1,m])$ is finite for all $m\in \N\setminus\{0\}$, where $[1,m]=\{1,\ldots,m\}$.
 \end{enumerate} Then $\mathcal D$ is closed in $\mathcal X$ and homeomorphic to the Cantor set.
\end{proposition}
\begin{proof} Statement a) is classical, see \cite[2-14]{H-Y} and \cite{S}. The compactness comes from Tychonoff's theorem and the definition of the projective limit topology. To prove b) we begin by showing that $\mathcal D$ can be endowed with a metric compatible with the projective limit topology. Given $w,w'\in \mathcal D$ with $w\neq w'$, define $r(w,w')$ to be the smallest integer $n$ such that $b_{\infty,\iota}(w)\neq b_{\infty,\iota}(w')$ for some $\iota\in h^{-1}([1,n])$. This is the smallest element in a non-empty set of integers since if $w\neq w'$, by definition of the projective limit there is an index $\iota_0$ such that $b_{\infty,\iota_0}(w)\neq b_{\infty,\iota_0}(w')$. Thus, our set contains $h(\iota_0)$. It follows from the definition that given three different $w,w',w"$ we have that $r(w,w")\geq{\rm min}(r(w,w'),r(w',w"))$, and that $r(w,w')=r(w',w)$. We can define a distance on $\mathcal D$ by setting $d(w,w)=0$ and $d(w,w')= r(w,w')^{-1}$ for $w'\neq w$. By the definition of the projective limit topology on $\mathcal D$ we see that it is totally discontinuous because the $D_\iota$ are and that every ball $\B(w,\eta)=\{w'/d(w',w)\leq\eta\}$ is the intersection $\bigcap_{\iota\in h^{-1}([1,\lfloor\eta^{-1}\rfloor])}b_{\infty,\iota}^{-1}(w)$ of finitely many open sets . The first assumption implies that every ball of positive radius centered in a point of $\mathcal D$ is infinite, so that $\mathcal D$ is perfect. Finally, our space $\mathcal D$ is a perfect compact and totally disconnected metric space and thus homeomorphic to the Cantor set by \cite[Corollary 2-98]{H-Y}. The fact that $\mathcal D$ is closed in $\mathcal X$ follows from the fact that each $D_\iota$ is closed in $X_\iota$.\end{proof}
\begin{remark} In \cite[Theorem 2-95]{H-Y} it is shown that a compact totally disconnected metric space is homeomorphic to the projective limit of a projective system of finite discrete spaces. The authors then show that if two such spaces are perfect, the projective systems can be chosen so that their projective limits are homeomorphic.
\end{remark}
 We recall some definitions and facts of toric geometry that are needed for our purpose, referring to \cite{E} for proofs. Let $M=\Z^r$ be the lattice of integral points in $\R^r$ and $N={\rm Hom}_\Z(M,\Z)$ its dual, a lattice in $N_\R=\check\R^r$.  We assume that $r\geq 2$.\par A (finite) fan is a finite collection $\Sigma=(\sigma_\alpha)_{\alpha\in A}$ of rational polyhedral strictly convex cones in $N_\R$ such that if $\tau$ is a face of a $\sigma_\alpha\in\Sigma$, then $\tau$ is a cone of the fan, and the intersection $\sigma_\alpha \cap\sigma_\beta$ of two cones of the fan is a face of each. A rational polyhedral convex cone is by definition the cone positively generated by finitely many vectors of the lattice $N$, called integral vectors. It is strictly convex if it does not contain any non zero vector space. Given a rational polyhedral cone $\sigma$, its convex dual $\check\sigma=\{u\in \R^r/<u,v>\geq 0 \ \forall v\in\sigma\}$, where $<u,v>=v(u)\in\R$,  is again a rational polyhedral convex cone, which is strictly convex if and only if the dimension of $\sigma$, that is, the dimension of the smallest vector subspace of $\check\R^r$ containing $\sigma$ is $r$. A refinement $\Sigma'$ of a fan $\Sigma$ is a fan such that every cone of $\Sigma'$ is contained in a cone of $\Sigma$ and the union of the cones of $\Sigma'$ is the same as that of $\Sigma$. We denote this relation by $\Sigma'\prec\Sigma$.\par
By a theorem of Gordan (see \cite[Chap. V, \S 3, Lemma 3.4]{E}), for every rational strictly convex cone in $N_\R$ the semigroup $\check\sigma\cap M$ is a finitely generated semigroup generating $M$ as a group. If we fix a field $k$ the semigroup algebra $k[t^{\check\sigma\cap M}]$ is finitely generated and corresponds to an affine algebraic variety $T_\sigma$ over $k$, which may be singular but is normal because the semigroup $\check\sigma\cap M$ is saturated in the sense that if for some $k\in\N_{>0}$ and $m\in M$ we have $km\in\check\sigma\cap M$, then $m\in\check\sigma\cap M$. To each fan $\Sigma$ is associated a normal algebraic variety $T_\Sigma$ obtained by glueing up the affine toric varieties ${\rm Spec} k[t^{\check\sigma\cap M}],\ \sigma\in\Sigma ,$ along the affine varieties corresponding to faces that are the intersections of two cones of the fan. A refinement $\Sigma'$ of a fan $\Sigma$ gives rise to a proper birational map $T_{\Sigma'}\to T_\Sigma$.\par\noindent
Each toric variety admits a natural action of the torus $T_{\{0\}}=(k^*)^r=(k-{\rm points \ of})\ \break{\rm Spec} k[t^ M]$ which has a dense orbit corresponding to the cone $\sigma=\{0\}$ of the fan. There is an inclusion reversing bijection between the cones of the fan and the orbits of the action of the torus on $T_\Sigma$. In an affine chart $T_\sigma$, the traces of the orbits of the torus action correspond to \textit{faces} $F_\tau$ of the semigroup $\check\sigma\cap M$: they are the intersections of the semigroup with the linear duals $\tau^\perp$ of the faces $\tau\subset\sigma$. The monomial ideal of $k[t^{\check\sigma\cap M}]$ generated by the monomials $t^\delta; \delta\notin F_\tau$ is prime and defines the closure of the orbit corresponding to $\tau$. When it is of maximal dimension the cone $\sigma$ itself corresponds to the zero dimensional orbit in $T_\sigma$ and all zero dimensional orbits are obtained in this way. A prime monomial ideal of $k[t^{\check\sigma\cap M}]$ defines the intersection with $T_\sigma$ of an irreducible subvariety invariant by the torus action.\par\medskip
A preorder on $M$ is a binary relation $\preceq$ with the properties that for any $m,n,o\in M$ either $m\preceq n$ or $n\preceq m$ and $m\preceq n\preceq o$ implies $m\preceq o$.\par\noindent An additive preorder on $M$ is a preorder $\preceq$ such that if $m\preceq m'$ then $m+n\preceq m'+n$ for all $n\in M$. It is a fact (see \cite{R}, \cite{Ku}, \cite{E-I}) that given any additive preorder $\preceq$ on $M$ there exist an integer $s,\ 1\leq s\leq r$ and $s$ vectors $v_1,\ldots ,v_s$ in $N_\R$ such that $$m\preceq n\ \hbox{\rm if and only if }(<m,v_1>,\ldots ,<m,v_s>)\leq_{lex}(<n,v_1>,\ldots ,<n,v_s>) ,$$ where $lex$ means the lexicographic order. An additive order is an additive preorder which is an order. This means that the vector subspace of $N_\R$ generated by $\nu_1,\ldots ,\nu_s$ is not contained in any rational hyperplane.\par\noindent
In accordance with the notations of \cite{GP-T} we denote by $w$ a typical element of $ZR(\Sigma)$ and by $m\preceq_w n$ the corresponding binary relation on $\Z^r$. Following Ewald-Ishida in \cite{E-I} we define a topology on the set of additive preorders as follows : 
\begin{definition}Let $\sigma$ be a rational polyhedral cone in $N_\R$. Define ${\mathcal U}_\sigma$ to be the set of additive preorders $w$ of $M$ such that $0\preceq_w \check\sigma\cap M$. The ${\mathcal U}_\sigma$ are a basis of open sets for a topology on the set $ZR(M)$ of additive preorders on $M$. Given a fan $\Sigma$ the union $ZR(\Sigma)=\bigcup_{\sigma\in\Sigma}{\mathcal U}_\sigma\subset ZR(M)$ endowed with the induced topology is defined by Ewald-Ishida as the Zariski-Riemann manifold of the fan $\Sigma$. By \cite[Proposition 2.10]{E-I}, it depends only on the support $\vert\Sigma\vert=\bigcup_{\sigma\in\Sigma}\sigma$ and if we assume $\vert\Sigma\vert=N_\R$, it is equal to $ZR(M)$. 
\end{definition}
This topology is the same as the topology defined in \cite{Ku} and \cite{S}, where a pre-basis of open sets is as follows: given two elements $a,b\in M$ a set of the pre-basis is the set ${\mathcal U}_{a,b}$ of preorders $w$ for which $a\preceq_w b$. Indeed, to say that a preorder $w$ is in the intersection $\bigcap_i{\mathcal U}_{a_i,b_i}$ of finitely many such sets is the same as saying that $\check\sigma\cap M\subset\{m\in M/0\preceq_w m\}$ where $\sigma\subset N_\R$ is the rational polyhedral cone dual to the cone $\check\sigma$ in $M_\R$ generated by the vectors $b_i-a_i$. If $\check\sigma=\R^r$, the intersection is the trivial preorder, where all elements are equivalent. \par
\begin{theorem}\label{EI}{{\rm (Ewald-Ishida in \cite[Theorem 2.4]{E-I})}} The space $ZR(M)$ is quasi-compact, and for any finite fan $\Sigma$ the space $ZR(\Sigma)$ is quasi-compact.
\end{theorem}
It is shown in \cite[Proposition 2.6]{E-I} that to any fan $\Sigma$ and any preorder $w\in ZR(\Sigma)$ we can associate a cone $\sigma\in\Sigma$. It is the unique cone with the properties $0\preceq_w \check\sigma\cap M$ and $\sigma^\perp\cap M=\{m\in\check\sigma\cap M/m\preceq_w 0\ {\rm and}\ m\succeq_w 0\}$. Following Ewald-Ishida, we say that $w$ dominates $\sigma$. If $w$ is an order, $\sigma$ is of maximal dimension $r$.\par
It is not difficult to verify that if $\Sigma'$ is a refinement of $\Sigma$ we have $ZR(\Sigma')=ZR(\Sigma)$ and moreover, given $w\in ZR(\Sigma)$ the corresponding cones $\sigma',\sigma$ verify $\sigma'\subset\sigma$, so that we have a torus-equivariant map $T_{\sigma'}\to T_\sigma$ of the corresponding toric affine varieties. \par
Given a fan $\Sigma$ its finite refinements, typically denoted by $\Sigma'$, form a directed partially ordered set. It is partially ordered by the refinement relation $\Sigma'\prec\Sigma$. It is a directed set because any two finite fans with the same support have a common finite refinement (see \cite[Chap. III]{KKMS} or \cite[Chap. VI]{E}). We are going to study three projective systems of sets indexed by it. The set $\{T_{\Sigma'}\}$ of torus-invariant irreducible subvarieties of $T_{\Sigma'}$, with the topology induced by the Zariski topology, the set $O_{\Sigma'}$ of the torus orbits of $T_{\Sigma'}$, again endowed with the Zariski topology, and finally the set of $0$-dimensional orbits, which is the set of closed points of $O_{\Sigma'}$. The first two sets are in fact equal because a torus invariant irreducible subvariety of $T_\Sigma$ is the closure of an orbit. This is the meaning of \cite[Lemma 3.3]{GP-T}: a prime monomial ideal of $k[t^{\check\sigma\cap M}]$ is generated by the monomials that are not in a \textit {face} of the semigroup, and the Lemma states that faces are the $\tau^\perp\cap \check\sigma\cap M$, where $\tau$ is a face of $\sigma$ and so corresponds to an orbit.\par The sets $\{T_{\Sigma'}\}=O_{\Sigma'}$ are finite since our collection of cones in each $\Sigma'$ is finite. The topology induced by the Zariski topology of $T_{\Sigma'}$ on $\{T_{\Sigma'}\}$ is such that the closure of an element of $\{T_{\Sigma'}\}$ is the set of orbits contained in the closure of the corresponding orbit of the torus action. \par\noindent Thus, the closed points of $\{T_{\Sigma'}\}$ are the zero dimensional orbits of the torus action on $T_{\Sigma'}$ and are in bijective correspondence with the cones of maximal dimension of $\Sigma'$.\par\noindent
Given a refinement $\Sigma'\prec\Sigma$, the corresponding proper birational equivariant map $T_{\Sigma'}\to T_\Sigma$ maps surjectively $\{T_{\Sigma'}\}$ to $\{T_{\Sigma}\}$ and zero dimensional orbits to zero dimensional orbits. The map induced on zero dimensional orbits is surjective because every orbit contains zero dimensional orbits in its closure.\par
Given an additive preorder $w\in ZR(\Sigma)$ it dominates a unique cone $\sigma'$ in each refinement $\Sigma'$ of $\Sigma$ and defines a unique torus-invariant irreducible subvariety of $T_{\Sigma'}$ corresponding to the prime ideal of $k[t^{\check\sigma'\cap M}]$ generated by the monomials whose exponents are $\succeq_w 0$. This defines a map

$$Z\ \colon \ ZR(\Sigma)\longrightarrow\limproj_{\Sigma'\prec\Sigma} \{T_{\Sigma'}\}.$$
We now quote two results from \cite{GP-T}:
\begin{theorem}\label{Zar}{{\rm (Gonz\'alez P\'erez-Teissier in \cite[Proposition 14.8]{GP-T})}}\par\noindent For any finite fan $\Sigma$, the map $Z$ is a homeomorphism between $ZR(\Sigma)$ with its Ewald-Ishida topology and $\limproj_{\Sigma'\prec\Sigma} \{T_{\Sigma'}\}$ with the projective limit of the topologies induced by the Zariski topology.\end{theorem}
\begin{corollary}\label{OZar}{{\rm (Gonz\'alez P\'erez-Teissier in \cite[Section 13]{GP-T})}} If $\vert\Sigma\vert=N_\R$ the map $Z$ is a homeomorphism between the space $ZR(M)$ of additive preorders on $\Z^r$ and $\limproj_{\Sigma'\prec\Sigma} \{T_{\Sigma'}\}$ and induces a homeomorphism between the space of additive orders on $M$ and the projective limit of the discrete sets $\{T_{\Sigma'}\}_0$ of $0$-dimensional orbits.
\end{corollary}
\begin{definition}\label{height} The height $h(\Sigma)$ of a finite fan $\Sigma$ in $N_\R$ is the maximum absolute value of the coordinates of the primitive vectors in $N$ generating the one-dimensional cones of $\Sigma$. There are only finitely many fans of height bounded by a given integer.
\end{definition}
We note that the fan consisting of the $2^r$ quadrants of $N_\R$ and their faces has height one.\par
Let us show that we can apply Proposition \ref{or} to obtain the conjunction of Sikora's result in \cite[Proposition 1.7]{S} and Ewald-Ishida's in \cite[Proposition 2.3]{E-I}:
\begin{proposition}For $r\geq 2$ the space of additive preorders on $Z^r$ is quasi-compact. It contains as a closed subset the space of additive orders, which is homeomorphic to the Cantor set.
\end{proposition}
\begin{proof}To apply Proposition \ref{or}, our partially ordered directed set is the set of finite fans in $N_\R$ with support $N_\R$. The set of additive preorders is quasi-compact by Proposition \ref{or}. Concerning the set of orders, we apply the second part of the proposition. Our discrete sets are the sets $\{T_{\Sigma'}\}_0$ of zero dimensional torus orbits of the toric varieties $T_{\Sigma'}$. Each refinement of a cone of maximal dimension contains cones of maximal dimension and since $r\geq 2$ each cone $\sigma'\in\Sigma'$ can be refined into infinitely many fans with support $\sigma'$ which produce as many refinements of $\Sigma'$. To see this, one may use the fact that each cone of maximal dimension, given any integral vector in its interior, can be refined into a fan whose cones are all regular and which contains the cone generated by the given integral vector (see \cite[Chap. III]{KKMS} or \cite[Chap. VI]{E}). Our function $h(\Sigma')$ is the height of definition \ref{height} above, which has the required properties since $h(\Sigma')\geq h(\Sigma)$ if $\Sigma'\prec\Sigma$ because the one dimensional cones of $\Sigma$ are among those of $\Sigma'$.
\end{proof}
\begin{remark}There  are of course many metrics compatible with the topology of the space of orders induced by that of $ZR(M)$ and another one is provided following \cite[\S 1]{Ku} and \cite[Definition 1.2]{S}: Let $\B(0,D)$ be the ball centered at $0$ and with radius $D$ in $\R^r$. Define the distance $\tilde d(w,w')$ to be $0$ if $w=w'$ and otherwise $\frac{1}{D}$, where $D$ is the largest integer such that $w$ and $w'$ induce the same order on $\Z^r\cap\B(0,D)$. It would be interesting to verify directly, perhaps using Siegel's Lemma (see \cite{Sc}), that the distances $d(w,w')$ and $\tilde d(w,w')$ define the same topology.
\end{remark}
\begin{remark}The homeomorphism $Z$ of Theorem \ref{Zar} is the toric avatar of Zariski's homeomorphism between the space of valuations of a field of algebraic functions and the projective limit of the proper birational models of this field. As explained in \cite[\S 13]{GP-T} the space of orders is the analogue for the theory of preorders of zero dimensional valuations in the theory of valuations. 
\end{remark}

\begin{acknowledgement}
I am grateful to Hussein Mourtada for interesting discussions of the first topic of this note and for calling my attention to the phenomenon described in the case $p=2$.
\end{acknowledgement}

\end{document}